\documentclass[12pt]{article}

\usepackage{amsmath}
\usepackage{amsthm,amscd,amsfonts}
\usepackage{amssymb, upref, color,epstopdf}
\usepackage{amsmath,amssymb,amsthm}
\usepackage{color}
\usepackage[colorlinks]{hyperref}
\usepackage[dvips]{graphicx}
\usepackage{epsf,epsfig,subfigure, verbatim}
\usepackage{latexsym,bm}

\numberwithin{equation}{section}

\theoremstyle{plain}
\newtheorem{exam}{Example}[section]
\newtheorem{theorem}[exam]{Theorem}
\newtheorem{lemma}[exam]{Lemma}
\newtheorem{remark}[exam]{Remark}

\newtheorem{definition}[exam]{Definition}

\begin{document}
\date{}

%\newpage

\title{ Limit cycles of a Li\'enard system with symmetry allowing for discontinuity
}
\author{Hebai Chen$^1$\footnote{Hebai Chen was supported by  NNSF of China grant (No. 11572263).},~~Maoan Han$^{2,3}$\footnote{Maoan Han was supported by National Natural Science Foundation of China (11431008 and
11771296).},
\,\,\,\,\,\, Yong-Hui Xia$^{4,5}$
 \footnote{ Corresponding author. Email: xiadoc@163.com.  Yonghui Xia was supported by the National Natural
Science Foundation of China under Grant (No. 11671176 and No. 11271333), Natural
Science Foundation of Zhejiang Province under Grant (No. LY15A010007), Marie Curie Individual Fellowship within the European Community Framework Programme(MSCA-IF-2014-EF, ILDS - DLV-655209), the Scientific Research Funds of Huaqiao University and China Postdoctoral Science Foundation (No. 2014M562320). }
%\,\,\,\,\, Weinian Zhang$^{3}$\footnote{Weinian Zhang was supported by....}
\\
 %{\small 1. School of Mathematics and Statistics, Jiangsu Normal University Xuzhou, 221116 Jiangsu, China.}\\
{\small 1.College of Mathematics and Computer Science, Fuzhou University, Fuzhou, Fujian 350002, P. R. China}\\
{\small \em  chen\_hebai@sina.com (H. Chen) }
 \\
 {\small 2. School of Mathematical Sciences, Shanghai Normal University, Shanghai, Shanghai, China.}\\
 {\small \em  mahan@shnu.edu.cn (M. Han)}
  \\
 {\small 3. School of Mathematical Sciences, Qufu Normal University, Qufu, 273165, P.R.China.}
 \\
 {\small 4. School of Mathematical Sciences, Huaqiao University, 362021, Quanzhou, Fujian, China.}\\
{\small\em xiadoc@163.com  (Y-H.Xia)}\\
{\small \em 5.Department of Mathematics, Zhejiang Normal University, Jinhua, 321004, China}
}

 \maketitle

\begin{center}
\begin{minipage}{140mm}
\begin{abstract}
 This paper presents new results on the limit cycles of a Li\'enard system with symmetry allowing for discontinuity.
Our results generalize and improve the results in \cite[Theorem 1 and 2]{Zhang} or
 the monograph \cite[Chapter 4, Theorem 5.2]{Zh}. The results in \cite{Zh} are only valid for the smooth system.
We emphasize that our main results are valid for discontinuous systems.
Moreover, we show the presence and an explicit upper bound for the amplitude of the {\bf two limit cycles}, and we estimate the position of the double-limit-cycle bifurcation surface in
the parameter space. Until now, there is no result to determine the amplitude of the {\bf two limit cycles}. The existing results on the amplitude of limit cycles guarantee that the Li\'enard system has {\bf a unique limit cycle}.
Finally, some applications and examples are provided to show the effectiveness of our results. We revisit a co-dimension-3 Li\'enard oscillator (see \cite{LR,ZY2}) in Application 1. Li and Rousseau \cite{LR} studied the limit cycles of such a system when the parameters are small.
However, for the general case of the parameters (in particular, the parameters are large), the upper bound of the limit cycles remains open. We completely provide the bifurcation diagram for the one-equilibrium case.  Moreover, we determine the amplitude of the two limit cycles and estimate the position of the double-limit-cycle bifurcation surface for the one-equilibrium case. Application 2 is presented to study the limit cycles of a class of the Filippov system.

\end{abstract}



\end{minipage}
\end{center}

\section{Introduction and  main results}
The Hilbert 16th problem was proposed as one of 23 famous problems in mathematics in 1900. It remains open until now. The Hilbert 16th problem has two parts: classification of the on ovals, which are defined by a polynomial equation $\{H(x, y) = 0\}$, and the number of limit cycles of polynomial vector fields. In this paper, we focus on the problems related to the
second part(\cite{DRR, Kaloshin}). The most important topic of the second part is to find the upper bounded number of limit cycles, which is one of the main themes of the quantitative theory of ordinary differential equations (see, e.g., \cite{Chow1,Chow2,Chow3,Zhao-JMAA,DRR,DR,Han,HLZ,Kaloshin,LS,Li,Rychkov,Smale,TH,Han-JDE,Xiao-JMAA,Zhang}). Since the Hilbert 16th problem is notably difficult, it remains open (see \cite{Li}),
 Smale \cite{Smale} suggested to first solve the number of limit cycles of the polynomial Li\'enard systems.
 In fact, the Li\'enard system is a notably common system in engineering and can exhibit notably rich dynamics. %Smale also try to solve the number of limit cycles of the classic Li\'enard system.
  The investigation of Li\'enard systems has a long history and many results for the limit cycles (see \cite{Zh}). Rychkov \cite{Rychkov} studied a Li\'enard system as follows.
   \begin{eqnarray}
\left\{\begin{array}{l}
                   \displaystyle\frac{dx}{dt}=y-(\mu_1x+\mu_2x^3+x^5), \\
                   \displaystyle\frac{dy}{dt}=-x,
                 \end{array}\right.
\label{R}
\end{eqnarray}
where $(\mu_1,\mu_2)\in\mathbb{R}^2$. He obtained some sufficient conditions to guarantee that system \eqref{R} exists at most two limit cycles.
   Zhang (see \cite[Theorem 1]{Zhang}) generalized Rychkov's result to a general smooth Li\'enard system as follows:
\begin{eqnarray}
\left\{\begin{array}{l}
                    \displaystyle\frac{dx}{dt}=y-F(x), \\
                   \displaystyle\frac{dy}{dt}=-g(x),
                 \end{array}\right.
\label{initial}
\end{eqnarray}
where $f(x)$ is a continuous function, and $F(x)=\int_0^xf(s)ds$. %$F(x)=-F(-x)$,  $g(x)=-g(-x)$ and $xg(x)>0$ for $\forall x\ne0$.
They proved that system \eqref{initial} had at most two limit cycles under suitable conditions. Zhang {\it et al.} also collected this important theorem in the monograph on the quantitative theory (see \cite{Zh}).
For comparison, we restate their results (\cite[Theorems 1 and 2]{Zhang} or \cite[Chapter 4, Theorems 5.1 and 5.2]{Zh}) as follows.\\
\smallskip

\noindent {\bf Theorem A}{  (\cite[Theorems 1 and 2]{Zhang} or \cite[Chapter 4, Theorems 5.1 and 5.2]{Zh})} {\em
Considering the system
\eqref{initial},
if the following conditions hold:
\begin{description}
\item[(a)] $F'(x)\in C^0(-d,d)$ for $d>0$ and $F(-x)=-F(x)$;
\item[(b)] $F(x)=0$ (resp. $<0$, $>0$) for $x=\beta_1,\beta_2$ (resp. $x\in(\beta_1,\beta_2)$, $x\in(0,\beta_1)\cup(\beta_1,d)$),
$F'(x)=0$ (resp. $\leq0$) for $x=\alpha_1$ (resp. $\beta_1<x<\alpha_1$), where $0<\beta_1<\alpha_1<\beta_2<d$;
\item[(c)] $g(x)$ is Lipschitz continuous in $(-d, d)$, $xg(x)>0$ for $\forall x\neq0$, $g(-x)=-g(x)$ in $(-d,d)$
and $G(-\infty)=G(+\infty)=+\infty$, where $G(x)=\int_0^xg(s)ds$;
\item[(d)] either $f(x)$ or $f(x)/g(x)$ is nondecreasing for $x\in[\alpha_1,d]$.
 \end{description}
 Then, system \eqref{initial} has at most two limit cycles.
 (In other words, system \eqref{initial} has either two simple limit cycles or one semi-stable limit cycle if the limit cycle(s) exists(exist).)
 }\\
%\label{mainresultZh}
%\end{theorem}

%\begin{remark}

In \cite{Rychkov}, Rychkov assumed that $f(x)\in C^1(-d,d)$, $(\frac{f(x)}{x})'>0$
for $x\geq\alpha_1$, and $g(x)=x$. Moreover, Rychkov \cite{Rychkov} set the requirements that $f(x)$ has exactly two positive zeros,
and $f(x)$ may have infinitely many positive zeros in $(0,\beta_1)$. In fact,
Theorem A enables $f(0)=0$. Thus, Theorem A (Zhang \cite{Zhang}) improves Rychkov's theorem.
A question is as follows:
Are all conditions in Theorem A sharp? Is it possible to reduce the conditions of Theorem A?

Meanwhile,
many engineering devices can be modeled as nonsmooth dynamical systems,
which deserve considerable attentions, and one can refer to \cite{BBCK} and references therein. In general, the theory of the smooth dynamical system cannot be directly applied to the discontinuous case. In fact, when the vector field of \eqref{initial} is discontinuous, there may be
grazing solutions, sliding solutions or impact solutions. Thus, we must study the discontinuous dynamic systems in theory and applications.
 The analysis shows that system \eqref{initial} has no sliding solutions. Therefore, how to extend Theorem A (Zhang \cite{Zhang} smooth system) to the discontinuous case is an important problem. In this paper, we present more general results that can be applied to
the discontinuous Li\'enard system.

%\end{remark}

%For smooth Li\'enard systems, there are many results on the nonexistence, existence and uniqueness of limit cycles, see \cite{Zh}.
%However, there are few results about smooth Li\'enard systems having at most two limit cycles, not to mention, the discontinuous (nonsmooth) dynamical systems, see \cite{Zh}.
For Li\'enard systems, there are many results on the nonexistence, existence and uniqueness of limit cycles;
see \cite{Zh}.
However, there are no related results on the discontinuous (nonsmooth) Li\'enard systems with at most two limit cycles (see \cite{Zh}).
Thus, we present the following generalized Zhang-type theorem, which enables discontinuity.

\begin{theorem}
Corresponding to system
\eqref{initial},
assume that the following conditions hold:
\begin{description}
\item[(H1)] $F(x)$ is Lipschitz continuous in $(-d,d)$ for $d>0$ and $F(-x)=-F(x)$;
\item[(H2)] $F(x)=0$ (resp. $<0$, $>0$) for $x=\beta_1,\beta_2$ (resp. $x\in(\beta_1,\beta_2)$, $x\in(0,\beta_1)\cup(\beta_1,d)$),
$F(x)$ is $C^1$ continuous in $(\beta_1, \alpha_1)\cup(\alpha_1,d)$,  and  $F'(x)\leq0$ for $\beta_1<x<\alpha_1$, where $0<\beta_1<\alpha_1<\beta_2<d$;
\item[(H3)]  $g(x):=g_0(x)+c~{\rm sign}(x)$, where $g_0$ is Lipschitz continuous in $(-d, d)$, $xg_0(x)>0$ for $\forall x\neq0$,  $g_0(-x)=-g_0(x)$ in $(-d,d)$ and $c\geq0$;
\item[(H4)] either $f(x)$ or $(F(x)-F(\alpha_1))f(x)/g(x)$ is nondecreasing for $x\in[\alpha_1,d]$.
 \end{description}
 Then, system \eqref{initial} has at most two limit cycles.
 (In other words, system \eqref{initial} has either two simple limit cycles or one semistable limit cycle if the limit cycle(s) exists(exist).)
\label{mainresult}
\end{theorem}

\begin{remark}
System \eqref{initial} can be discontinuous at $x=0$ in our theorem (Theorem \ref{mainresult}). We will apply our results to a generalized Filippov system, which is a discontinuous system. Application 2 in Section 5 is provided to show that our result is valid for the discontinuous system.
\end{remark}

\begin{remark}
As observed, condition {\bf (H4)} of Theorem \ref{mainresult} is weaker than condition {\bf (d)} of Theorem A. Example 1 in Section 5 shows that in some bad situations, Theorem A is invalid, but our theorem works.
\end{remark}

\begin{remark}
We also remove the condition $G(-\infty)=G(+\infty)=+\infty$ ($G(x)=\int_0^xg(s)ds$) in {\bf (c)} of Theorem A.
\end{remark}

\begin{remark}
Application 1 in Section 5 is provided to show the feasibility of our result. We revisit a generalized codimension-3 Li\'enard oscillator, which was considered in \cite{LR,ZY2}. In fact, Li and Rousseau \cite{LR} studied the limit cycles of such a system when the parameters are small.
However, for the general case of the parameters (in particular, the parameters are large), the upper bound of the limit cycles remains open. Therefore, we give a complete bifurcation diagram of the co-dimension-3 Li\'enard oscillator for the one-equilibrium case.
\end{remark}

Another important topic in the quantitative theory of differential equation is to find the relative position and amplitude of the limit cycles.
  However, few papers considered the amplitude of the limit cycles for smooth Li\'enard systems (\cite{Alsholm,CL,JiangFF,Odani,YZ}) and discontinuous (non-smooth) systems.
%It should be noted that the Li\'enard systems considered in \cite{Alsholm,CL,Odani,YZ} are smooth and symmetry.
 Moreover, the existing results on the amplitude of limit cycles were to guarantee that the Li\'enard system had a unique limit cycle.
 Until now, as far as we know, no paper considered the amplitude of the two limit cycles for the smooth Li\'enard systems and non-smooth dynamical systems.
  The next theorem is presented to guarantee the presence of two limit cycles. We also study the amplitude of the two limit cycles. %An explicit upper bound for the amplitude is given.

\begin{theorem}
In addition to all conditions in Theorem \ref{mainresult}, suppose that $f(x)/g(x)$ is nondecreasing for $x\in(\alpha_1,d)$ and $\beta_2\geq 2\beta_1$.
If we further assume that
there exists $\xi\in[\beta_1,\beta_2]$
and  $F(x)+F(x+\xi)<0$ for $(0,\xi)$,
then,
  system \eqref{initial} has exactly two limit cycles when
  $\int_{\beta_1}^{d_1}g(x)F(x)dx\geq0$
  and at least one limit cycle when $\int_{\beta_1}^{d_1}g(x)F(x)dx<0$.
\label{mainresult2}
\end{theorem}

\begin{remark}
  Theorem \ref{mainresult2} provides sufficient conditions for the presence of exactly two limit cycles and the explicit upper bound for the amplitude of the two limit cycles. To the best knowledge of the authors, there is no existing result related to the amplitude of the two limit cycles.
  Moreover, Theorem \ref{mainresult2} can help us to estimate the position of the double-limit-cycle bifurcation surface in the parameter space.
\end{remark}
\begin{remark}
To show the effectiveness of Theorem \ref{mainresult2}, in Application 1 in Section 5, we determine the amplitude of the two limit cycles for a generalized co-dimension-3 Li\'enard oscillator, and we estimate the position of the double-limit-cycle bifurcation surface for this case.
\end{remark}

The remainder of the paper is organized as follows.
The preliminary results on local results and the criterion of multiplicity and stability of limit cycles are provided in Section 2.
Sections 3 and 4 present the proofs of Theorems \ref{mainresult} and \ref{mainresult2}, respectively. Finally, applications and examples are provided to show the effectiveness of our result. A complete bifurcation diagram of a generalized co-dimension-3 Li\'enard oscillator in the entire parameter space is shown in Application 1. Application 2 is presented to study the limit cycles of a class of the Filippov system, which is discontinuous.

\section{Preliminaries}
In this section, to prove the main result, we need the following preliminary lemmas.

\begin{lemma}
In the interval $|x|\leq d$, the initial value problem of system \eqref{initial} has a unique solution,
and the origin is a sink.
\label{unique}
\end{lemma}
\begin{proof}
We prove this lemma with two cases.\\
\noindent{\bf Case 1.} If  $c=0$, i.e., $g(0)=0$, $g(x)$ is clearly Lipschitzian continuous for $x\in(-d,d)$.
Hence, the conclusion of this lemma holds.

\noindent{\bf Case 2.} If $c>0$, $\lim_{x\to0^+}g(x)=c>0$, we have $\lim_{x\to0^-}g(x)=-c<0$ by $g(-x)=-g(x)$.
In other words,
\[
\Sigma=\{(x,y)\in\mathbb{R}^2|x=0\}
\]
is a discontinuity boundary.
 Now, we recall the Filippov convex method (\cite{Filippov,KRG}). We can construct general solutions based on the standard solutions in regions $x<0$ and $x>0$
and sliding solutions on $\Sigma$.
Let
\begin{eqnarray*}
\delta=((H_x,H_y),(y-F(x),-g(x))|_{x\to0^-}\cdot((H_x,H_y),(y-F(x),-g(x))|_{x\to0^+}=y^2,
\end{eqnarray*}
where $H(x,y)=x$.
Therefore, the crossing set is
\[
\Sigma_c=\{(x,y)\in\mathbb{R}^2|x=0, y\neq0\}.
\]
Now, we must discuss the origin.
Since $(0,0)$ lies in the discontinuous line $x=0$,
there is no Jacobian matrix of $(y-F(x),-g(x))$ at $(0,0)$.
Let
\begin{eqnarray}
\mathcal{E}(x,y)=\int_0^xg(s)ds+\frac{y^2}{2}=\int_0^xg_0(s)ds+c|x|+\frac{y^2}{2}.
\label{Exy}
\end{eqnarray}
It is clear that
\begin{eqnarray}
\frac{d\mathcal{E}}{dt}=-g(x)F(x)<0,~{\rm for}~0<x<\beta_1.
\label{dExy}
\end{eqnarray}
Now, we give an assertion: any orbit $\varphi(t;0,y_0)$ of Eq. \eqref{initial} either returns to a point $(0,y_1)$ %$(y_1>0)$
%on the positive $y$-axis
after time $t_0$ or directly approaches the origin, where $y_0,y_1$
are small, and $0<y_1<y_0$.
Then, to prove that the origin is a sink, it suffices to show the above assertion.
Consider the following auxiliary Hamiltonian system
\begin{eqnarray}
\left\{\begin{array}{l}
                  \displaystyle \frac{dx}{dt}=y, \\
                  \displaystyle \frac{dy}{dt}=-g(x).
                 \end{array}\right.
\label{Hami}
\end{eqnarray}
Let
$\phi(t;0,y_0)$ be an orbit of Eq. \eqref{Hami}, which passes through the same initial point $(0,y_0)$.
$\phi(t;0,y_0)$ is clearly a closed orbit.
By comparing the theorem and the signs of systems \eqref{initial} and \eqref{Hami},
we have the following: the positive orbit $\varphi(t;0,y_0)$ lies in the interior of $\phi(t;0,y_0)$ when $t>0$.
The above assertion is proven. Consequently, the proof of this lemma is complete.
%Hence,$\varphi(t;0,y_0)$ of system \eqref{initial} either returns to a point $(0,y_1)$ in the positive $y$-axis after time to or directly approaches the origin.
\end{proof}

To prove Theorem \ref{mainresult}, we must have the following lemma to guarantee the multiplicity and stability of limit cycles for system \eqref{initial}  when $g(x)$ may not be $C^1$ and even discontinuous at $x=0$.
\begin{lemma}
Suppose that $F(x)$  satisfies $F'(x)\in C^0(-d,d)$, and $g(x)$ satisfies condition (H3).
Assume that system \eqref{initial} has a limit cycle $\gamma$.
Furthermore, if the following inequality holds,
\begin{eqnarray}
\oint_{\gamma}{\rm div}(y-F(x),-g(x))dt<0\,\,\,({\rm resp}. >0),
\label{int}
\end{eqnarray}
then, $\gamma$ is a simple limit cycle, which is stable. On the contrary, if
\[
\oint_{\gamma}{\rm div}(y-F(x),-g(x))dt>0,
\]
then, $\gamma$ is a simple limit cycle, which is unstable.
\label{multi}
\end{lemma}
\begin{proof}
By Lemma \ref{unique}, the initial value problem of system \eqref{initial} has a unique solution, although $g(x)$ is discontinuous at $x=0$.

Taking the following transformation
\begin{eqnarray}
u=h(x):={\rm sgn}(x)\sqrt{2\int_0^xg(s)ds}, ~~~
d\tau :=\frac{g(x)}{u}dt,
\label{tran}
\end{eqnarray}
we rewrite Eq. \eqref{initial} as follows
\begin{eqnarray}
\left\{
\begin{array}{l}
\displaystyle \frac{du}{d\tau} =y-\hat F(u), \\
\displaystyle\frac{dy}{d\tau}= -u,
\end{array}\right.
\label{simeq}
\end{eqnarray}
where $\hat F(u):=F(h^{-1}(u))$.
On the one hand, \eqref{tran} is a homeomorphic transformation except for the $y$-axis.
Suppose that Eq. \eqref{initial} has a limit cycle $\gamma$.
By the transformation \eqref{tran}, %we can suppose that
$\gamma$ can be changed into $\hat\gamma$ of \eqref{simeq}.
Thus, the stability of $\gamma$ is equivalent to $\hat\gamma$.
On the other hand, the vector field of \eqref{simeq} is clearly $C^1$ for $c=0$.
For $c\neq0$, the vector field of \eqref{simeq} is Lipschitz continuous, and it is $C^1$  except for the line $x=0$.
According to the transformation \eqref{tran} and \eqref{int2},
\begin{eqnarray}
\oint_{\hat\gamma}{\rm div}(y-\hat F(u),-u)d\tau=\oint_{\gamma}{\rm div}(y-F(x),-g(x))dt<0,\,\,({\rm resp}. >0).
\label{int2}
\end{eqnarray}
By Theorem 2.2 of \cite[Chapter 4]{Zh},
$\hat\gamma$ is a stable (resp. unstable) and simple limit cycle by \eqref{int2}.
Consequently, $\gamma$ is a stable (resp. unstable) and simple limit cycle.
 The proof is complete.
\end{proof}

\begin{remark}
It is difficult to provide the criterion to determine the stability of limit cycles for a general discontinuous system, even if there is a one discontinuous line for a general vector field. For example, %it is not easy to determine the stability of a limit cycle for the piecewise smooth systems (\cite{LHZ},\cite{TH}).
Liang et al. \cite{LHZ} showed a criterion to determine the multiplicity and stability of limit cycles for planar piecewise smooth integrable systems.
However, the main purpose of this paper is to study the limit cycles of the Li\'enard system \eqref{initial} allowing for discontinuity.
 Thus, our aim of Lemma \ref{multi} is to give the criterion to determine the multiplicity and stability of limit cycles for system \eqref{initial}  when $g(x)$ may not be $C^1$, even discontinuous at $x=0$. In this case, we find the relationship between the stability and the divergence \eqref{int}.
We also remark that we do not require that either $g(x)$ or $F(x)$ is an odd function.
 \end{remark}

 In the following, we discuss the number of limit cycles in $[-\alpha_1,\alpha_1]$.
\begin{lemma}
In the interval $|x|\leq \alpha_1$,  there is at most one limit cycle system of \eqref{initial}.
Moreover, this limit cycle is unstable if it exists.
\label{ulc}
\end{lemma}
\begin{proof}
\begin{figure}[htp]
\centering
\scalebox{0.6}[0.6]{
\includegraphics{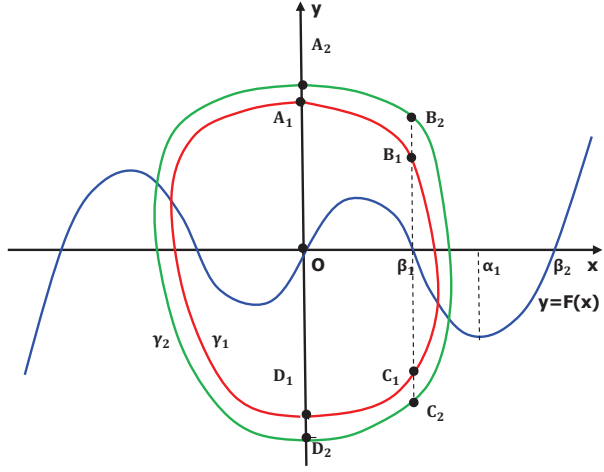}}
\caption{Discussion about limit cycles in $-\alpha_1\leq x \leq \alpha_1$ }
\label{lc1}
\end{figure}

Using contradiction, if it is not true, we can suppose that there are at least two limit cycles of system \eqref{initial} in the interval  $-\alpha_1\leq x \leq \alpha_1$,
where $\gamma_1$ and $\gamma_2$ are any two such limit cycles and $\gamma_1$ lies in the interior of $\gamma_2$. See Figure \ref{lc1}.
To proceed with the proof, firstly, we assert that there is no limit cycle in the interval  $-\beta_1\leq x \leq \beta_1$ for system \eqref{initial}.
If it is not the case, suppose that there is a limit cycle $\gamma$ of  system \eqref{initial} in this interval. Then,
it follows that
\begin{eqnarray}
\oint_{\gamma}d\mathcal{E}=0,
\label{E}
\end{eqnarray}
where $\mathcal{E}(x,y)$ is defined as \eqref{Exy}.
On the other hand, in view of \eqref{Exy}, we have $\oint_{\gamma}d\mathcal{E}<0$.
This contradicts \eqref{E}. Thus, the assertion holds. There is no limit cycle of system \eqref{initial} in $x\in[-\beta_1,\beta_1]$.

Note that the vector field $(y-F(x),-g(x))$ is symmetrical about the origin.
Therefore, we have
\begin{eqnarray}
\int_{\widehat{A_1B_1D_1}}d\mathcal{E}=\frac{1}{2}\oint_{\gamma_1}d\mathcal{E},~~
\int_{\widehat{A_2B_2D_2}}d\mathcal{E}=\frac{1}{2}\oint_{\gamma_2}d\mathcal{E}.
\label{equal0}
\end{eqnarray}
Next, we will show that the following inequality holds.
\begin{eqnarray}
\int_{\widehat{A_1B_1D_1}}d\mathcal{E}<
\int_{\widehat{A_2B_2D_2}}d\mathcal{E}.
\label{nonequal0}
\end{eqnarray}
The proof idea follows \cite{LS}.
We represent the orbit segments $\widehat{A_1B_1}$ and $\widehat{A_2B_2}$ by $y=y_1(x)$
and $y=y_2(x)$, respectively.
It follows that
\begin{eqnarray*}
\int_{\widehat{A_iB_i}}d\mathcal{E}=\int_0^{\beta_1}\frac{-g(x)F(x)}{y_i(x)-F(x)}dx,
\end{eqnarray*}
where $i=1,2$.
Thus,
\begin{eqnarray*}
\int_0^{\beta_1}\left(\frac{-g(x)F(x)}{y_1(x)-F(x)}-\frac{-g(x)F(x)}{y_2(x)-F(x)}\right)dx= \int_0^{\beta_1}\frac{-g(x)F(x)(y_2(x)-y_1(x))}{(y_1(x)-F(x))(y_2(x)-F(x))}dx<0.
\end{eqnarray*}
In other words, we have
\begin{eqnarray}
\int_{\widehat{A_1B_1}}d\mathcal{E}<
\int_{\widehat{A_2B_2}}d\mathcal{E}.
\label{nonequal01}
\end{eqnarray}
Similarly, we have
\begin{eqnarray}
\int_{\widehat{C_1D_1}}d\mathcal{E}<
\int_{\widehat{C_2D_2}}d\mathcal{E}.
\label{nonequal02}
\end{eqnarray}
We represent the orbit segments $\widehat{B_1C_1}$ and $\widehat{B_2C_2}$ by $x=x_1(y)$
and $x=x_2(y)$, respectively.
It is clear that
\begin{eqnarray*}
\int_{\widehat{B_iC_i}}d\mathcal{E}=\int_{y_{B_i}}^{y_{C_i}}F(x_i(y))dy,
\end{eqnarray*}
where $i=1,2$.
Since $F(x)<0$ in $(\beta_1,\alpha_1)$, $y_{B_1}<y_{B_2}$
and $y_{C_1}>y_{C_2}$, we have
\begin{eqnarray*}
\int_{y_{B_2}}^{y_{C_2}}F(x_2(y))dy>\int_{y_{B_1}}^{y_{C_1}}F(x_2(y))dy.
\end{eqnarray*}
Note that $x_2(y)>x_1(y)$ for $y_1\leq y\leq y_2$.
With the monotonic decrease in $F(x)$,
$F(x_2(y))<F(x_1(y))$ when $y_1\leq y \leq y_2$.
Hence, we can obtain
\begin{eqnarray}
\int_{y_{B_1}}^{y_{C_1}}F(x_1(y))dy-\int_{y_{B_1}}^{y_{C_1}}F(x_2(y))dy=\int_{y_{B_1}}^{y_{C_1}}(F(x_1(y))-F(x_2(y)))dy<0.
\label{nonequal03}
\end{eqnarray}
By (\ref{nonequal01}-\ref{nonequal03}), it follows that the inequality (\ref{nonequal0}) holds.
However, since $\gamma_1$ and $\gamma_2$ are limit cycles, we have
\begin{eqnarray*}
\frac{1}{2}\oint_{\gamma_1}d\mathcal{E}=\frac{1}{2}\oint_{\gamma_2}d\mathcal{E}=0,
\end{eqnarray*}
which, combined with (\ref{equal0}), lead to
\begin{eqnarray*}
\int_{\widehat{A_1B_1D_1}}d\mathcal{E}=
\int_{\widehat{A_2B_2D_2}}d\mathcal{E}=0.
\end{eqnarray*}
This contradicts the inequality (\ref{nonequal0}).
Thus, the proof of this lemma is complete.
\end{proof}

The proof idea of Lemma \ref{ulc} comes from \cite{LS}. We discuss the number of limit cycles in $[-d,d]$ in the following lemma.

\begin{lemma}
For system \eqref{initial}, there are at most two limit cycles that intersect $x=\alpha_1$ in the interval $|x|\leq d$.
\label{twolc}
\end{lemma}

\begin{figure}[htp]
\centering
\scalebox{0.55}[0.55]{
\includegraphics{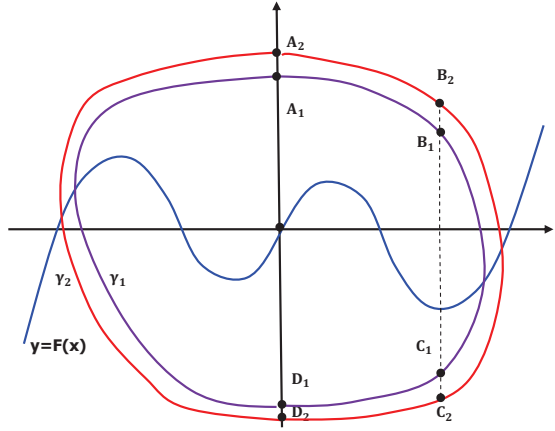}}
\caption{Discussion about the limit cycles in $-d\leq x \leq d$}
\label{lc2}
\end{figure}

\begin{proof}

Using contradiction, if Lemma \ref{twolc} is not true, then we can suppose that there are at least two limit cycles: $\gamma_1$ and $\gamma_2$. Without loss of generality, we suppose that $\gamma_1$ lies in the interior of $\gamma_2$ (see Figure \ref{lc2}).
Then, we can prove that the following inequality holds.
 \[
 \oint_{\gamma_1}{\rm div}(y-F(x),-g(x))dt> \oint_{\gamma_2}{\rm div}(y-F(x),-g(x))dt,
 \]
 or using another notation,
 \begin{eqnarray}
 \oint_{\gamma_1} F'(x)dt< \oint_{\gamma_2}F'(x)dt.
 \label{two}
 \end{eqnarray}
 %where $\gamma_1$, $\gamma_2$ are the two limit cycles  (see Figure \ref{lc2}) and $\gamma_1$ lies in the interior of $\gamma_2$.
In fact, because the vector fields of system \eqref{initial} are symmetrical about the origin,
we have
\begin{eqnarray}
&&\oint_{\gamma_1}F'(x)dt=2\int_{ \widehat{A_1B_1D_1}}F'(x)dt,
\\
 &&\oint_{\gamma_2}F'(x)dt=2\int_{ \widehat{A_2B_2D_2}}F'(x)dt.
  \label{two1}
\end{eqnarray}
Thus, to prove the inequality \eqref{two}, it suffices to show
\begin{eqnarray}
\int_{ \widehat{A_1B_1D_1}}F'(x)dt<\int_{ \widehat{A_2B_2D_2}}F'(x)dt.
 \label{two2}
\end{eqnarray}
By similar arguments to \cite{Zhang}, we can prove the following two inequalities:
\begin{eqnarray}
\int_{ \widehat{A_1B_1}}F'(x)dt<\int_{ \widehat{A_2B_2}}F'(x)dt,
\\
\int_{ \widehat{C_1D_1}}F'(x)dt<\int_{ \widehat{C_2D_2}}F'(x)dt
 \label{two3}
\end{eqnarray}
%Thus, we omit the proof of \eqref{two3}.
By condition {\bf (H4)} and \cite[Lemma 4.5 of Chapter 4]{Zh} or \cite[Theorem 1]{DR},
it follows that
\begin{eqnarray}
\int_{ \widehat{B_1C_1}}F'(x)dt<\int_{ \widehat{B_2C_2}}F'(x)dt.
 \label{two4}
\end{eqnarray}
Therefore, according to (\ref{two1}-\ref{two4}), inequality \eqref{two} holds. However, by a similar proof to the inequality \eqref{nonequal0}, we have
\begin{eqnarray*}
\int_{\widehat{A_1B_1D_1}}d\mathcal{E}<
\int_{\widehat{A_2B_2D_2}}d\mathcal{E}.
\end{eqnarray*}
Consequently,
\[
 \oint_{\gamma_1}F'(x)dt> \oint_{\gamma_2}F'(x)dt,
 \]
 which contradicts \eqref{two}. Therefore, it is not true that system \eqref{initial} has at least two limit cycles, i.e.,
 the assertion of Lemma \ref{twolc} is true.
\end{proof}

\section{Proof of Theorem \ref{mainresult}}

%In this section, we will give the proof of Theorem \ref{mainresult} by the aforementioned lemmas.

{\bf Proof of Theorem \ref{mainresult}}

Based on the number of limit cycles lying  $[-\alpha_1,\alpha_1]$,
we divide the proof of Theorem \ref{mainresult} into the following two cases.
%Based on the existence or nonexistence of limit cycles lying  $[-\alpha_1,\alpha_1]$,

{\bf Case 1:} system \eqref{initial} has a limit cycle $\gamma_0$ lying $[-\alpha_1,\alpha_1]$.
Suppose that there are at least two limit cycles that intersect $x=\alpha_1$ for system \eqref{initial},
where $\gamma_1$ is the innermost limit cycle
and $\gamma_2$ is the outermost one.
According to the instability of $\gamma_0$, $\gamma_1$ is internally stable.
Considering Lemma \ref{multi},
we have
\[
\oint_{\gamma_1}F'(x)dt\geq0.
\]
Meanwhile, $\gamma_2$ is externally stable. According to Lemma \ref{multi},
we see that
\[
\oint_{\gamma_2}F'(x)dt\geq0.
\]
Moreover, there are two stable (unstable) limit cycles surrounding the origin and adjacent.
By (\ref{two}),
we have
\[
\oint_{\gamma_2}F'(x)dt>0.
\]
Moreover, system \eqref{initial} has no limit cycle between $\gamma_1$ and $\gamma_2$,
and we have
\[
\oint_{\gamma_1}F'(x)dt=0.
\]
Consider the following system
\begin{eqnarray}
\left\{\begin{array}{l}
                   \dot{x}=y-F_1(x), \\
                   \dot{y}=-g(x),
                 \end{array}\right.
\label{initial2}
\end{eqnarray}
where $F_1(x)=F(x)+\kappa r(x)$, $\kappa>0$ is sufficiently small and
\begin{eqnarray}
r(x)=
\left\{\begin{array}{l}
                  0,\ \ \ \ \ \ \ \ \ \ \ \ \ \ \ \ \ \ \ \ \ \ {\rm for}~~|x|<\alpha_1, \\
                  {\rm sgn}(x)(|x|-\alpha_1)^3,~~{\rm for}~~|x|\geq\alpha_1.
                 \end{array}\right.
\label{initial2}
\end{eqnarray}
Since system \eqref{initial2} satisfies the conditions for the uniqueness of solutions to initial value problems according to Lemma \ref{unique},
system \eqref{initial2} is a generalized rotated vector field about $\kappa$.
System \eqref{initial2} clearly reduces system \eqref{initial} as $\kappa=0$.
Thus, $\gamma_2$ splits into at least two limit cycles $\tilde \gamma_2$
and $\hat\gamma_2$, where $\tilde \gamma_2$
lies in the interior of $\hat\gamma_2$.
By \cite[Theorem 2.2 of Chapter 4]{Zh}, it follows that
\begin{eqnarray*}
\oint_{\tilde\gamma_1}F'(x)dt\geq0, \ \
\oint_{\hat\gamma_1}F'(x)dt\leq0,
\end{eqnarray*}
which contradicts \eqref{two}.
Hence, in this case, system \eqref{initial2} has at most one limit cycle intersecting $x=\alpha_1$.

{\bf Case 2:} system \eqref{initial} has no limit cycle in $[-\alpha_1,\alpha_1]$.
Since the origin is a sink and all orbits that cross any points at infinity are repelling,
 system \eqref{initial} has $0$ or even limit cycles intersecting $x=\alpha_1$.
 Hence, assume that system \eqref{initial} has even limit cycles intersecting $x=\alpha_1$,
 where $\gamma_1$ is the innermost limit cycle, and $\gamma_2$ is the outermost one.
 Thus, $\gamma_1$ is internally unstable, and $\gamma_2$ is externally stable.
In view of Lemma \ref{multi}, we have
  \begin{eqnarray*}
\oint_{ \gamma_1}F'(x)dt\leq0, \ \
\oint_{\gamma_2}F'(x)dt\geq0.
\end{eqnarray*}
According to \eqref{two}, system \eqref{initial} has at most three limit cycles $\gamma_1$, $\gamma_2$, and $\gamma_3$,
where
\begin{eqnarray*}
 \oint_{ \gamma_1}F'(x)dt<0, \ \oint_{ \gamma_3}F'(x)dt=0,  \
 \oint_{\gamma_2}F'(x)dt>0.
\end{eqnarray*}
It is clear that $\gamma_3$ is semistable.
Here, we omit the proof because it is identical to the previous case.
We complete the proof of Theorem \ref{mainresult}.
$\hfill{} \Box$

\section{Proof of Theorem \ref{mainresult2}}

To prove this theorem, we first recall the following definition and theorem of \cite{Zh} .
\begin{definition} (\cite[p.302]{Zh})
The two curves $y=F_1(x)$ and $y=F_2(x)$ satisfy the following conditions:
\begin{description}
\item[(1)] $y=F_1(x)$ and $y=F_2(x)$ have $n+2$ intersection points $(a_i,b_i)$, where $i=1,\ldots n+2$, $a=a_1<a_2<\ldots<a_{n+1}<a_{n+2}=b$
and $(-1)^{i+1}[F_2(x)-F_1(x)]\geq0$ for $a_i<x<a_{i+1}$.
\item[(2)] There are $\tau_{i+1}^j, \xi_{i+1}^j\in[a_{i+1},a_{i+2}]$ and $\xi_{i+1}^j\geq \tau_{i+1}^j$ such that
\subitem(i) $(-1)^{i+j}F_j(x)\geq0$ for $x\in[\tau_{i+1}^j,\gamma_{i+1}^j]\subset[a_{i+1},a_{i+2}]$,
 \subitem(ii) $(-1)^i[(-1)^jF_j(x)+(-1)^lF_l(x+{\bar \Delta}_{i+1}^l]\geq0$ and $\not\equiv0$ for $x\in[a_i,\tau_{i+1}^j]$,
 \\
 where $\Delta_{i+1}^l:=\tau_{i+1}^j-a_i$, ${\bar \Delta}_{i+1}^l:=\xi_{i+1}^j-a_i$, $\gamma_{i+1}^j:=\max_{j=1,2}(\xi_{i+1}^j+\Delta_{i+1}^j)$, $j\neq l$,
 $j,l=1,2$, $i=1,2,\ldots,n+1$.
\end{description}
Then, $y=F_1(x)$ and $y=F_2(x)$ are $n$-fold mutually inclusive in $[a,b]$.
\label{fmi}
\end{definition}
Now, we consider the following system
\begin{eqnarray}
\left\{\begin{array}{l}
                   \dot{x}=\varphi(y)-\tilde F(x), \\
                   \dot{y}=-\tilde g(x).
                 \end{array}\right.
\label{GLD}
\end{eqnarray}
Corresponding to system \eqref{GLD}, we make the following assumptions.
\begin{description}
\item[(a)] $\varphi(y), \tilde g(x), \tilde F(x)\in C^0(-d,d)$ for large $d>0$
and \eqref{GLD} satisfy the conditions for the uniqueness of solutions.
\item[(b)] $x\tilde g(x)>0$ for $x\neq0$, $\tilde g(x)$ is odd, and $\tilde g(x)$ is nondecreasing.
\item[(c)] $y\varphi(y)>0$ for $y\neq0$,   $\varphi(y)$ is increasing and $\lim{y\to\infty}\varphi(y)=\infty$.
\end{description}

For convenience, we restate a theorem in \cite{Zh} as follows.

\noindent {\bf Theorem B} (\cite[Theorem 5.9]{Zh})
Assume that the conditions {\bf (a-c)} hold, and $\tilde F(x)$ and $\tilde F(-x)$ are $n$-fold mutually inclusive in $[0,b]$.
Then, system \eqref{GLD} has at least $n$ limit cycles, where they intersect $[a_i,a_{i+1}]$.
%\label{thm7.9}

{\bf
Proof of Theorem \ref{mainresult2}}
Since there are $\xi\in[\beta_1,\beta_2]$
and  $F(x)+F(x+\xi)<0$ for $(0,\xi)$,
according to Definition \ref{fmi}, $F(x)$ and $-F(-x)$ are $1$-fold mutually inclusive in the interval $[0,\beta_2]$.
According to Theorem B, system \eqref{initial} has at least one limit cycle in the interval $(-\beta_2,\beta_2)$.
Moreover, we can obtain that $y_H+y_I<0$.
$F(x)$ and $-F(-x)$ are not $2$-fold mutually inclusive in the interval $[0,x_0]$ for $\forall x_0>0$, i.e., we cannot obtain that system \eqref{initial} has at least two limit cycles according to Theorem B.

\begin{figure}[htp]
\centering
\scalebox{0.55}[0.55]{
\includegraphics{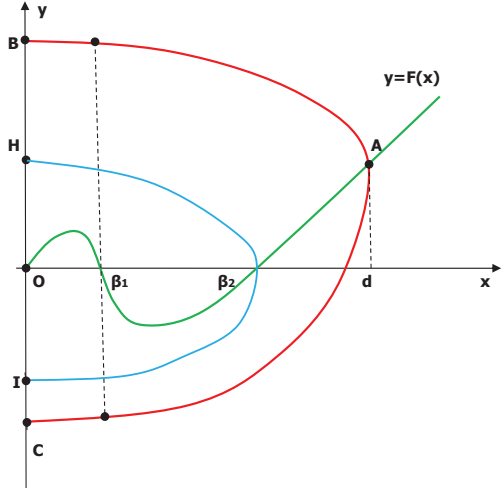}}
\caption{Discussion about  $y_B+y_C$}
\label{amplitude}
\end{figure}

Then, we discuss that $\int_{\beta_1}^dg(x)F(x)dx\geq0$.
Let $y=y_1(x)$ and $y=y_2(x)$ represent $\widehat{BA}$ and $\widehat{CA}$, respectively.
On the one hand,
for $x\in(\beta_1,\beta_2)$, it follows that
\[
y_1(x)-F(x)>y_1(\beta_2)-F(x)>y_1(\beta_2)-F(\beta_2).
\]
On the other hand,
for $x\in(\beta_2,d)$, it follows that
\[
y_1(x)-F(x)<y_1(\beta_2)-F(x)<y_1(\beta_2)-F(\beta_2).
\]
Let $v:=y_2(x)-F(x)$.
Then, we have
\[
\frac{dv}{dx}=-\frac{g(x)}{v}-f(x).
\]
It is clear that
\[
-\frac{g(x)}{v}-f(x)>0
\]
for $x\in(\beta_1,\alpha_1)$, as shown in Figure \ref{vx1}.
When $x\in(\beta_2,d)$, we claim that $v(x)=y_2(x)-F(x)$ lies upon $v=-g(x)/f(x)$.
We assume that $v(x)=y_2(x)-F(x)$ has a intersection point with $v=-g(x)/f(x)$.
Since $v=-g(x)/f(x)$ increases, $v(x)=y_2(x)-F(x)$ cannot intersect the $v$-axis.
This is a contradiction. Thus, $v(x)=y_2(x)-F(x)$ is increasing.
In other words, $v(x)<v(\beta_2)$ for $x\in(\beta_1,\beta_2)$
and $v(x)>v(\beta_2)$ for $x\in(\beta_2,d)$.

\begin{figure}[htp]
\centering
\scalebox{0.55}[0.55]{
\includegraphics{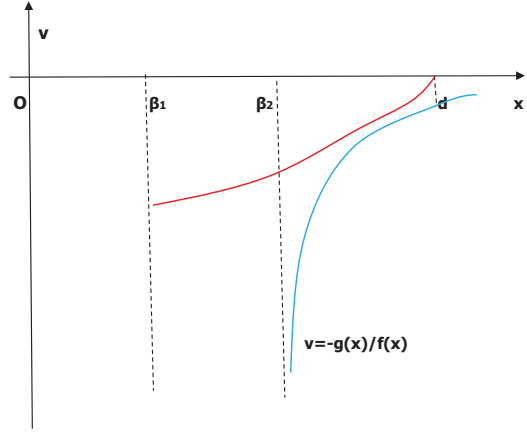}}
\caption{Discussion about the increase of $v(x)$}
\label{vx1}
\end{figure}

Let $u(x):=y_1(x)-F(x)$.
According to \eqref{Exy},
\begin{eqnarray*}
 \frac{d^2}{2}-\frac{\beta_1^2}{2}-\frac{y_1^2(\beta_1)}{2} &=&  \int_{\beta_1}^d\frac{-g(x)F(x)}{u(x)}dx
 \\
&=&  \int_{\beta_1}^{\beta_2}\frac{-g(x)F(x)}{u(x)}dx+\int_{\beta_2}^d\frac{-g(x)F(x)}{u(x)}dx
 \\
&>&  \int_{\beta_1}^{\beta_2}\frac{-g(x)F(x)}{u(\beta_2)}dx+\int_{\beta_2}^d\frac{-g(x)F(x)}{u(\beta_2)}dx
\\
&=&  -\frac{1}{u(\beta_2)}\int_{\beta_1}^dg(x)F(x)dx.
\end{eqnarray*}
Similarly, we can show that
\begin{eqnarray*}
\frac{\beta_1^2}{2}+\frac{y_2^2(\beta_1)}{2}- \frac{d^2}{2}
&=&  \frac{1}{v(\beta_2)}\int_{\beta_1}^dg(x)F(x)dx.
\end{eqnarray*}
Thus, we have
 \begin{eqnarray}
\frac{y_2^2(\beta_1)}{2}- \frac{y_1^2(\beta_1)}{2}<\left(\frac{1}{v(\beta_2)}-\frac{1}{u(\beta_2)}\right)\int_{\beta_1}^dg(x)F(x)dx\leq0.
\label{beta1}
\end{eqnarray}
Meanwhile, according to \eqref{Exy}, we have
 \begin{eqnarray}
\frac{\beta_1^2}{2}+\frac{y_1^2(\beta_1)}{2} -\frac{y_1^2(0)}{2} &=&   \int_0^{\beta_1}\frac{-g(x)F(x)}{u(x)}dx<0,
\nonumber\\
\frac{y_2^2(0)}{2}-\frac{\beta_1^2}{2}-\frac{y_2^2(\beta_1)}{2}  &=&   \int_0^{\beta_1}\frac{g(x)F(x)}{v(x)}dx<0.
\label{beta2}
  \end{eqnarray}
  According to \eqref{beta1} and \eqref{beta2},
  \[
  \frac{y_2^2(0)}{2}-\frac{y_1^2(0)}{2}<0,
  \]
i.e., $y_1(0)+y_2(0)>0$.
According to the Poincar\'e-Bendixson theorem, system \eqref{initial} has another limit cycle.
This theorem is proven.

\section{Applications and Examples}

\subsection{  Application 1: Limit cycles of a generalized co-dimension-3 Li\'enard oscillator}

Consider the following generalized co-dimension-3 Li\'enard oscillator
\begin{eqnarray}
\left\{\begin{array}{l}
                   \dot{x}=y-(ax+bx^3+x^5), \\
                   \dot{y}=-(cx+x^3),
                 \end{array}\right.
\label{GLO}
\end{eqnarray}
where $(a,b,c)\in\mathbb{R}^3$ (see \cite{LR,ZY2}).
In this section, we only discuss that system \eqref{GLO} has a unique equilibrium, i.e., $c\geq0$.\\

 When $|a|,|b|,|c|$ are small (i.e., system \eqref{GLO} can be changed into a near-Hamiltonian system),
limit cycles have been studied by \cite{LR}. However, for the general case of the parameters $|a|,|b|,|c|$ (particularly when the parameters are large), the upper bound of the limit cycles remains open. Therefore, we will provide a complete bifurcation diagram of system \eqref{GLO} in the parameter space $\mathbb{R}^3$.  Moreover, as mentioned in Theorem \ref{mainresult2}, we determine the amplitude of the two limit cycles, and we estimate the position of the double-limit-cycle bifurcation surface in the parameter space.

\begin{lemma}
When $a\geq0$ and $b\geq-2\sqrt{a}$, system \eqref{GLO} has no limit cycle.
When $a<0$, or $a=0$ and $b<0$, system \eqref{GLO} has a unique limit cycle,
which is stable.
\label{glo1}
\end{lemma}
\begin{proof}
When $a\geq0$ and $b\geq-2\sqrt{a}$, we obtain $a+bx^2+x^4\geq0$ for $\forall x\in\mathbb{R}$.
According to \eqref{dExy},
\begin{eqnarray*}
 \frac{dE(x,y)}{dt}=-g(x)F(x)=-x^2(c+x^2)(a+bx^2+x^4)\leq0.
\end{eqnarray*}
Suppose that there is a limit cycle $\gamma$ for system \eqref{GLO}.
Then, we have
\[
\oint_{\gamma}dE=\oint_{\gamma}-x^2(c+x^2)(a+bx^2+x^4)dt<0.
\]
This equation contradicts $\oint_{\gamma}dE=0$.
Thus, system \eqref{GLO} has no limit cycle.

When $a<0$, or $a=0$ and $b<0$,
it follows that $a+bx^2+x^4$ has a unique positive zero.
All conditions of \cite{LS} are clearly satisfied.
Hence, system \eqref{GLO} has a unique limit cycle,
which is stable.
\end{proof}

\begin{lemma}
When $a>0$ and $b<-2\sqrt{a}$, system \eqref{GLO} has at most two limit cycles.
\label{glo2}
\end{lemma}

\begin{proof}
Since $F(x)=ax+bx^3+x^5$, condition (H1) clearly holds.
When $a>0$ and $b<-2\sqrt{a}$, $F(x)=ax+bx^3+x^5$ has exactly two positive zeros $\beta_1,\beta_2$,
where
\[
\beta_1=\sqrt{\frac{-b-\sqrt{b^2-4a}}{2}},~~\beta_2=\sqrt{\frac{-b+\sqrt{b^2-4a}}{2}}.
\]
It is also clear that conditions (H2) and (H3) hold.

Now, we can also compute that
$F'(x)=a+3bx^2+5x^4$ has exactly two positive zeros $ x_1, x_2$,
i.e.,
\[
 x_1=\sqrt{\frac{-3b-\sqrt{9b^2-20a}}{10}},~~ x_2=\sqrt{\frac{-3b+\sqrt{9b^2-20a}}{10}}.
\]
Hence, $F'(x)=5(x^2- x_1^2)(x^2- x_2^2)$.
It is obvious that
 $F''(x)g(x)-F'(x)g'(x)=5[x^6+(x_1^2+x_2^2)x^4-3x_1^2x_2^2]+5c[3x^4-(x_1^2+x_2^2)x^2-x_1^2x_2^2]>0$
 for $x>x_2$.
Thus,
 \begin{eqnarray*}
  \frac{d[(F(x)-F(x_2))F'(x)/g(x)]}{dx}=\frac{F'(x)^2}{g(x)}+\frac{(F(x)-F(x_2))[F''(x)g(x)-F'(x)g'(x)]}{g^2(x)}>0
 \end{eqnarray*}
 for $x>x_2$.
Consequently, condition (H4) holds.
Thus, system \eqref{GLO} has at most two limit cycles.
\end{proof}

\begin{figure}[htp]
\centering
\scalebox{0.66}[0.66]{
\includegraphics{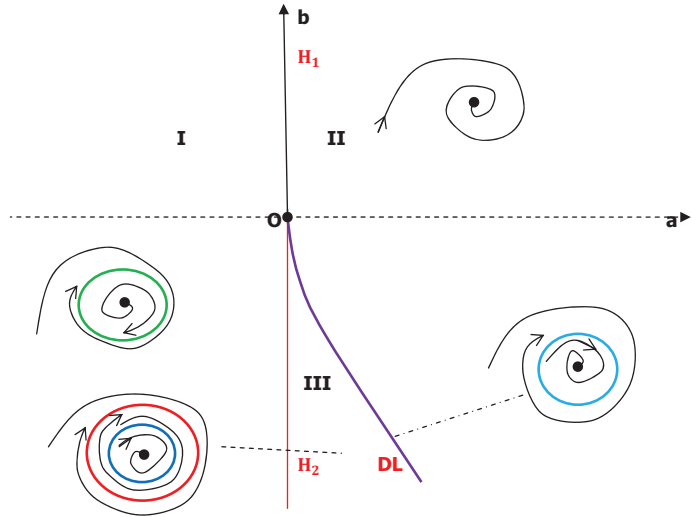}}
\caption{Bifurcation diagram of system \eqref{GLO} for a given $c>0$}
\label{GLObd}
\end{figure}

 Now we can state a result on the bifurcation diagram of system \eqref{GLO}.

\begin{theorem}
The bifurcation diagram of system \eqref{GLO} is shown in Figure \ref{GLObd},
where
\begin{eqnarray*}
I &=& \{(a,b,c)\in\mathbb{R}^2\times\mathbb{R}^+:  a<0\},
\\
II &=& \{(a,b,c)\in\mathbb{R}^2\times\mathbb{R}^+: a>0, b>\varphi(a)\},
\\
III &=& \{(a,b,c)\in\mathbb{R}^2\times\mathbb{R}^+: a>0, b<\varphi(a)\},
\\
H_1 &=& \{(a,b,c)\in\mathbb{R}^2\times\mathbb{R}^+: a=0, b\geq0\},
\\
H_2&=& \{(a,b,c)\in\mathbb{R}^2\times\mathbb{R}^+: a=0, b<0\},
\\
DL &=& \{(a,b,c)\in\mathbb{R}^2\times\mathbb{R}^+: b=\varphi(a,c)\},
\end{eqnarray*}
  $-5\sqrt{a}/2<\varphi(a,c)<-2\sqrt{a}$ and $\varphi(a,c)$ is a decreasing function about $a$.
\label{glopro}
\end{theorem}

\smallskip

\begin{proof}
When $c>0$, $O$ is clearly a source when $a<0$
and a sink when $a>0$.
Furthermore,
when $c>0$ and $a=0$,
we can check that
$O$ is a stable fine focus of order 1 when $b>0$
and an unstable fine focus of order 1 when $b<0$ according to \cite[p. 156]{GH}.
Moreover, when $c>0$ and $a=b=0$,
$O$ is a stable fine focus of order 2 according to Bautin bifurcation Theorem of \cite[Chapter 8]{Kuz}.
$O$ is a Bautin point.
Thus, $H_1$ and $H_2$ are two Hopf bifurcation surfaces for $c>0$.

When $c=0$ and $a\neq0$, using the transformation
\[
x\to x+\frac{1}{a}y,~y\to y,
\]
we rewrite system \eqref{GLO} as follows
\begin{eqnarray}
\left\{\begin{array}{l}
                   \dot{x}=-ax-\big(\frac{1}{a}+b\big)\big(x+\frac{1}{a}y\big)^3-\big(x+\frac{1}{a}y\big)^5,
                    \\
                   \dot{y}=-\big(x+\frac{1}{a}y\big)^3.
                 \end{array}\right.
\label{GLO2}
\end{eqnarray}
According to Theorem B.1 of \cite[Chapter 2]{Zh}, the origin of system \eqref{GLO2} is a stable degenerate node when $a>0$
and an unstable degenerate node when $a<0$, and so is the origin of system \eqref{GLO}.
When $a=c=0$, with the transformation of
\[
x\to x,~y\to y+bx^3+x^5,
\]
system \eqref{GLO} can be rewritten as
\begin{eqnarray}
\left\{\begin{array}{l}
                   \dot{x}=y,
                    \\
                   \dot{y}=-x^3-(3bx^2+5x^4)y.
                 \end{array}\right.
\label{GLO3}
\end{eqnarray}
According to Theorem B.2 of \cite[Chapter 2]{Zh}, the origin of system \eqref{GLO3} is a stable degenerate node, and so is the origin of system \eqref{GLO}.
Thus, $H_1$ and $H_2$ are two generalized Hopf bifurcation surfaces for $c=0$.

The fixed $c$ and $b$(resp. $a$) make it easy to check that system \eqref{GLO} is a generalized rotated vector field about $a$(resp. $b$).
When $a$(resp. $b$) increases, a stable limit cycle contracts, and an unstable limit cycle expands according to Theorem 3.5 of \cite[Chapter 4]{Zh}.
When $a=\epsilon$ and $b<0$,
system \eqref{GLO} has exactly two limit cycles, where $\epsilon>0$ is sufficiently small.
Moreover,
when $b=-2\sqrt{a}$, system \eqref{GLO} has no limit cycle.
Therefore, in $(\epsilon, b^2/4)$, there is a unique $a^*$ (denoted by $\phi(b,c)$)
such that system \eqref{GLO} has a semistable limit cycle.
Furthermore, system \eqref{GLO} has exactly two limit cycles when $0<a<\phi(b,c)$
and no limit cycle when $a>\phi(b,c)$.
Hence, $b=\varphi(b,c)=\phi^{-1}(b,c)$, i.e., $DL$ is a double-limit-cycle bifurcation surface.

Furthermore, we will prove that system \eqref{GLO} has exactly two limit cycles
when $b\leq-5\sqrt{a}/2$. $b\leq-5\sqrt{a}/2$ is clearly equivalent to $\beta_2\geq2\beta_1$.
Then, for $x\in(0,\beta_1)$, we can obtain
\begin{eqnarray*}
F(x)+F(x+\beta_1) &=& x(x^2-\beta_1^2)(x^2-\beta_2^2)+x(x+\beta_1)(x+2\beta_1)(x^2+2\beta_1x+\beta_1-\beta_2^2)
\\
 &=&x(x+\beta_1)[2x^3+3\beta_1x^2+(5\beta_1^2-2\beta_2^2)x+2\beta_1^3-\beta_1\beta_2^2]
 \\
 &\leq&x(x+\beta_1)(2x^3+3\beta_1x^2-3\beta_2^2x-2\beta_1^3)
  \\
 &=&x(x+\beta_1)(x-\beta_1)(6x^2+9\beta_1x+2\beta_1^2)<0.
\end{eqnarray*}
On the other hand, we obtain that
\begin{eqnarray*}
 \frac{d(f(x)/g(x))}{dx} &=&  \frac{x^6+(\alpha_1^2+\alpha_2^2)x^4-3\alpha_1^2\alpha_2^2x^2+c[3x^4-(\alpha_1^2+\alpha_2^2)x^2-\alpha_1^2\alpha_2^2]}{(cx+x^3)^2}>0
\end{eqnarray*}
for $x\geq \alpha_1$,
where $\alpha_1,\alpha_2$ are zeros of $f(x)$ and $0<\alpha_2<\alpha_1$.
Moreover, $\int_{\beta_1}^xF(s)g(s)ds>0$ when $x$ is large.
By Theorem \ref{mainresult2}, system \eqref{GLO} has exactly two limit cycles
when $b\leq-5\sqrt{a}/2$.
Thus, $\varphi(a,c)>-5\sqrt{a}/2$.
Consequently, this lemma is proven.
\end{proof}

Clearly, the bifurcation diagrams of systems \eqref{R} and \eqref{GLO} are similar.
Thus, we can also prove that the double-limit-cycle bifurcation curve $b=\varphi(a)$ of systems \eqref{R}
satisfies  $-5\sqrt{a}/2<\varphi(a)<-2\sqrt{a}$.\\

\subsection{ Application 2: limit cycles of a class of the Filippov system}

Consider the following generalized Filippov system, which is a discontinuous system
\begin{eqnarray}
\left\{\begin{array}{l}
                   \dot{x}=y-(ax+bx^3+x^5), \\
                   \dot{y}=-x-c\ {\rm sgn}(x),
                 \end{array}\right.
\label{Filippov}
\end{eqnarray}
where $(a,b,c)\in\mathbb{R}^2\times\mathbb{R}^+$.
See \cite{Chen,Kunze}.
 When $a\geq0$ and $b\geq-2\sqrt{a}$, system \eqref{Filippov} has no limit cycle;
when $a<0$, or $a=0$ and $b<0$, system \eqref{Filippov} has a unique limit cycle,
which is stable;
when $a>0$ and $b<-2\sqrt{a}$, system \eqref{Filippov} has at most two limit cycles.
Since the proofs are identical to those in Section 4, we omit them.
Theorem \ref{mainresult} has been applied in system \eqref{Filippov}.
Of course, for  system \eqref{Filippov}, we have the similar bifurcation diagram of Theorem \ref{glopro}.

\subsection{ Example 1}

Here, an example is presented to show that our results is valid for the non-smooth systems. This example also shows that Theorem \ref{mainresult} is more general than Theorem A in \cite{Zh} even if it reduces to a smooth system.

Consider the following piecewise linear system
\begin{eqnarray}
\left\{\begin{array}{l}
                   \dot{x}=y-{\rm sgn}(x)[a_1|x|^{2/3}+\frac{a_2}{2}(||x|^{2/3}+1|-||x|^{2/3}-1|)
                   \\~~~~~~~~~~~+\frac{a_3}{2}(||x|^{2/3}+2|-||x|^{2/3}-2|)], \\
                   \dot{y}=-x^{1/3},
                 \end{array}\right.
\label{PLS}
\end{eqnarray}
where $a_1>0$, $a_2>0$ and $-a_1-a_2<a_3<-a_1-a_2/2$.
It is clear that conditions (1-3) of Theorem \ref{mainresult} hold.
It is easy to verify that for $x>2\sqrt{2}$,
\begin{eqnarray*}
\frac{df(x)}{dx}&=& -\frac{a_1}{3x^{4/3}}<0,
\\
 \frac{d(f(x)/g(x))}{dx}&=& -\frac{2a_1}{3x^{5/3}}<0,
 \\
  \frac{d[(F(x)-F(2\sqrt{2}))f(x)/g(x)]}{dx}&=& 0.
\end{eqnarray*}
 Thus, condition (4) of Theorem \ref{mainresult} also holds. Therefore, our theorem is valid for system \eqref{PLS}. However, condition (d) of Theorem A in \cite{Zh} does not hold because both $f(x)$ and $ f(x)/g(x)$ decrease. Thus, Theorem A cannot be applied to this case even if system (\ref{PLS}) reduces to a smooth system.

%\section*{Acknowledgements}

\end{document}